\documentclass[11pt]{amsart}
\usepackage{amsmath,amsfonts,amssymb,amscd,verbatim,comment, amsthm}
\usepackage{tabularx}
\usepackage{fullpage}
\usepackage{color}
\usepackage{array}
\usepackage{lipsum}
\usepackage{enumerate}
\usepackage{mathtools}
\usepackage{longtable}
\usepackage{adjustbox}
\usepackage{tikz}
\usepackage{tikz-cd}
\usepackage{bbm}
\usetikzlibrary{positioning}

\allowdisplaybreaks
\theoremstyle{plain}
\newtheorem{thm}{Theorem} 
\newtheorem{theorem}{Theorem} 

\newtheorem{corollary}[thm]{Corollary} 
\newtheorem{lemma}[thm]{Lemma}
\newtheorem{proposition}[thm]{Proposition}

\theoremstyle{definition}
\newtheorem{definition}{Definition} 
 
\newtheorem{example}[thm]{Example}
 
\newtheorem{question}{Question}

\usepackage{xcolor}
\definecolor{opheliacolor}{RGB}{82,151,248}
\definecolor{danielcolor}{RGB}{82,248,151}

\DeclareMathOperator{\Gal}{Gal}

\title{Hilbertian Fields and Profinite Iterated Monodromy Groups Over Transcendental Elements with Specializations of Infinite Index}
\author{Daniel Gotshall}
\begin{document}
\begin{abstract}
Given a field $L$ and polynomial $f(x)/L$, one can consider the indices of the profinite iterated monodromy groups formed over the ground field within the profinite iterated monodromy groups over elements transcendental over $L$. For any countable Hilbertian field $k$, we show random Galois fixed field extensions almost surely make the indices infinite over all ground field elements, with certain conditions on $f(x)$. This is a generalization of a result by Dittmann and Kadets.
\end{abstract}
\maketitle
\section{Introduction}
\subsection{Background and Main Problem}
Given a polynomial $f(x)$ of degree $d>1$ over a field $L$, we let $f^n(x)$ denote the $n$-th composition of $f$ with itself. Similarly, the set of roots of $f^n(x)-c$ is written as $f^{-n}(c)$, where either $c=a\in L$ or $c=t$ is transcendental over $L$. Naturally, $f^{-\infty}(c):=\bigcup_{n=1}^\infty f^{-n}(c)$.
Clearly $f(x)-t$ is irreducible over $L(t)$. When $\text{char}(L)\nmid d$, it is also separable. Each element of the iterated Galois group over $t$ can be identified with an automorphism on a $d$-ary tree. This is done by letting the root vertex denote $t$ and the vertices on the $n$-th level be the elements of $f^{-n}(t)$. One can pass to the inverse limit of these groups with respect to $n$. Such objects are referred to as ``profinite iterated monodromy groups'' (abbreviated PFIMG). If $f(x)-a$ is no longer separable, one can simply pass to its largest separable factor $g_{L,f,a}(x)$ and replace $t$ with $a$ above. To make this precise, we use the following notation.
\begin{definition}
$\boldsymbol{\textbf{Aut}(T_\infty^d)}$ is the group of automorphisms on the infinite $d$-ary tree.
\end{definition}
\begin{definition}\label{Galois group notation}
Let $L$ be any infinite field and $f\in L[x]$ such that $\text{char}(L) \nmid \text{deg}(f)$. For each $n\geq 1$, let $g_{f,n}(x)$ be the largest separable factor of $f^n(x)-a$ over $L$.  Then we implement the following notation.
\begin{align*}
\boldsymbol{G_{f,n}}&:=\text{Gal}(L(f^{-n}(t))/L(t))    \\
\boldsymbol{G_{f}}&:=\text{Gal}(L(f^{-\infty}(t))/L(t))=\varprojlim G_{f,n}\\
\boldsymbol{G_{f,n,a}}&:=\text{Gal}(L(g_{f,n}^{-1}(a))/L)\\
\boldsymbol{G_{f, a}}&:=\varprojlim G_{f,n,a}
\end{align*}
\end{definition}
Let $P_0$ be the prime ideal $(t-a)$ in $L[t]$ and $P_n$  a prime ideal lying above $P_0$ in $L[f^{-n}(t)]$.
\begin{definition}We use the following terminology (shown to be well defined in \cite{Lang}).
\begin{itemize}
    \item $\boldsymbol{L_n}:=L[f^{-n}(t)]/P_n$ (so $L_0=L[t]/(t-a)\cong L$).
    \item $\boldsymbol{L_n^s/L_0}$ is the maximal separable subextension of $L_n/L_0$.
    \item $\boldsymbol{L_{n,a}}:=L(g_{f,n}^{-1}(a))$.
    \item $\boldsymbol{D_n}\leq G_{f,n}$ is the decomposition subgroup over $P_n$.
    \item $\boldsymbol{I_n}$ is the inertia subgroup over $P_n$ (it acts trivially on $L_n^s$ or equivalently all of $L_n$).
\item $\boldsymbol{J_n}\triangleleft D_n$ acts trivially on $L_{n,a}$ ($I_n\triangleleft J_n$).
\end{itemize}
\end{definition}
We have the following (see \cite{Lang}).
\begin{equation}\label{Decomp Subgroup n levels}
D_n/I_n\cong \text{Aut}(L_n/L_0)\cong \text{Gal}(L_n^s/L_0).   
\end{equation}
Often we get $L_n^s\cong L_{n,a}$. In these cases, we obtain $D_n/I_n\cong G_{f,n,a}$ by Equation \eqref{Decomp Subgroup n levels}. However, we can have $L_{n,a}\subsetneq L_n^s$ when $\text{char}(L)=p>0$, even for polynomials $f(x)/L$ such that $f^n(x)-t$ is separable and irreducible over $L(t)$. For an example when $n=1$, $p$ is odd, and $a=0$, consider $f(x)=x^{2p+1}-bx$ where $b^{1/2}, b^{1/p}\not\in L$ (such elements will always exist for the fields we study here). One can easily verify that $L_{1,0}=L$ and $[L_1^s:L]\geq 2$, meaning $I_1\lneq J_1$. Hence $G_{f,n,a}$ is a quotient of $D_n/I_n$ in general.
\begin{equation}\label{eq specialization n levels}
D_n/J_n\cong G_{f,n,a}.
\end{equation}

Definition \ref{Galois group notation} and Equation \eqref{eq specialization n levels} imply $G_{f,a}\cong \varprojlim D_n/J_n$. In the special case that $J_n$ is trivial (i.e. $f^n(x)-a$ is separable) $\forall \ n$, we can identify $G_{f, a}$ with the decomposition subgroup of $G_{f}$. In the more general setting, we abuse notation in a natural way as explained below.
\begin{definition}\label{Group Index Def}
\[[G_{f}:G_{f, a}]:=\lim_{n\rightarrow \infty}\frac{|G_{f,n}|}{|G_{f,n,a}|}= \lim_{n\rightarrow \infty} ([G_{f,n}:D_n]\cdot |J_n|).\]
\end{definition}
Notice the limit in Definition \ref{Group Index Def} is well defined because $|G_{f,n}|/|G_{f,n,a}|$ can only increase with $n$ (as $G_{f, n+1, a}$ projects onto $G_{f,n,a}$ for all $n$). Also, in the separable case previously described, Definition \ref{Group Index Def} gives the actual index of $G_{f, a}$ in $G_{f}$, as implied by a density argument.
\begin{theorem}[Odoni, Lemma 2.4 and Corollary 7.3 in \cite{Odoni}]\label{Subgps of Sd Wreath Prod}
Let $f(x)/L$ be a polynomial of degree $d>1$ such that $\text{char}(L)\nmid d$. Then $G_{f}\leq [S_d]_{i=1}^\infty\cong \text{Aut}(T_d^\infty)$. 
\end{theorem}
Definition \ref{Group Index Def} combined with Theorem \ref{Subgps of Sd Wreath Prod} imply that $[\text{Aut}(T_d^\infty):G_{f,a}]$ is well defined in an analogous manner.
\begin{question}\label{initial question}
For which (infinite) fields $L$, elements $a\in L$, and polynomials  $f\in L[x]$ do we have  \[{[G_{f}: G_{f, a}]}{<\infty}?\]
\end{question}
Notice we trivially get $[G_f: G_{f,a}]=\infty$ for any $a\in L$ whenever $L$ is algebraically closed. If $L$ is pseudo-algebraically closed, we similarly get a negative answer to Question \ref{initial question} for infinitely many $a\in k$. As all iterates of $f$ over $t$ are irreducible, $[G_f: G_{f,a}]$ is ``artificially inflated'' when the $n$-th iterates over $a$ are inseparable. Thus we modify Question \ref{initial question} to focus on Hilbertian fields $k$ (those satisfying Hilbert's Irreducibility Theorem) and points $a\in k$ outside a thin set.
\begin{question}\label{main problem}
For which Hilbertian fields $k$ and polynomials  $f(x)/k$ do we have  $[G_{f}: G_{f,a}]<\infty$ for all $a\in k\setminus \Omega$ for some thin set $\Omega$?
\end{question}
Odoni was the first to investigate Question \ref{main problem} upon restriction to $n$ iterated preimages.
\begin{theorem}[Odoni, \cite{Odoni}]\label{thm odoni finite iterate Gal gps}
Let $n\in\mathbb{N}$, $k$ be a Hilbertian field, and $f(x)/k$ be an irreducible polynomial. Then $\exists$ a thin subset $\Omega_n\subsetneq k$  such that $G_{f,n}\cong G_{f,n,a}$ $\forall \ a\in k\setminus \Omega_n$.
\end{theorem}
Note Theorem \ref{thm odoni finite iterate Gal gps} does not really address Question \ref{main problem} directly because infinite unions of thin sets are not necessarily thin. Serre showed that finite extensions of Hilbertian fields are also Hilbertian (see \cite{Serre}). Combining this with Hilbert's original Irreducibility Theorem stating that $\mathbb{Q}$ is Hilbertian, it is clear that all number fields are Hilbertian. Question \ref{main problem} has been answered for certain polynomials and rational functions over number fields of various degrees (see \cite{Tucker}, \cite{Benedetto_2018}, \cite{Jones}, \cite{Looper_2018}, \cite{Specter}, \cite{Stoll}). We state Specter's result due to its generality.
\begin{theorem}[Specter, \cite{Specter}] \label{Specter} Let $d>1$ and $L/\mathbb{Q}$ be an algebraic extension unramified outside finitely many primes. Then $\exists$ an infinite set $\{f_i\}_{i=1}^\infty$ of $f_i\in L[x]$ with degree $d_i$ such that $G_{f_i, 0}\cong G_{f_i}\cong \text{Aut}(T_\infty^{d_i})$.
\end{theorem}
Dittmann and Kadets showed Theorem \ref{Specter} cannot be extended to the case of ramification occurring over infinitely many primes.
\begin{theorem}[Dittmann-Kadets, \cite{DK}]\label{DK}
Let $k$ be a countable Hilbertian field of characteristic 0. Then $\exists$ a pseudo-algebraically closed Hilbertian extension $F/k$ such that for all polynomials $f(x)/F$ of degree $d\geq 2$, we have  $[\text{Aut}(T_\infty^d):G_{f,0}]=\infty.$
\end{theorem}
The proof of Theorem \ref{DK} uses a set of quadratic discriminant extensions coming from the index 2 subgroups of $\prod_{i=1}^\infty \mathbb{Z}/2\mathbb{Z}$. By lifting to the preimage of a natural group homomorphism, these correspond to degree 2 Galois extensions of $K$, where $K$ is a finite extension of $k$ over which $f$ is defined. Dittmann and Kadets then show that $F$ contains infinitely many of these minimal extensions to finish the proof.\\

We generalize some of these ideas to show that one actually obtains the same infinite index result over ``almost any'' randomly selected Galois fixed field extension in a natural probabilistic sense. Moreover, we describe cases where we get the stronger result that $[G_f: G_{f,a}]=\infty$. 
Namely, we only require that the finite extensions of $k$ have infinitely many minimal Galois extensions, by which we mean Galois extensions not containing any nontrivial, proper Galois subextensions. Also, we replace a technical result by Dittmann and Kadets on a family of vast field extensions (Theorem 1.4 in \cite{DK}) with an easier result on the fixed fields of certain subgroups. To make all this precise, one must know that all profinite groups have the Krull topology. Profinite groups also have the Haar measure, which can be normalized to form a probability measure. We shall also need the following definitions.
\begin{definition}[Common Notation]\label{Notation}
Let $L$ be any infinite field and $f\in L[x]$ any polynomial of degree greater than 1. We use the following notation throughout the rest of the paper.
\begin{itemize}
    \item $\mathbf{G_L}$ is the absolute Galois group over $L$.
    \item $\mathbf{G_f}$ \textbf{over} $\mathbf{L}$ is $\text{Gal}(L(f^{-\infty}(t))/L(t))$. Similarly, $\mathbf{G_{f,a}}$ \textbf{over} $\mathbf{L}$ is $\text{Gal}(L(f^{-\infty}(a))/L)$. (We may omit the phrase ``over $L$'' if the ground field is understood by context.)
    \item $\boldsymbol{\mu_L}$ is the normalized Haar measure on $G_L$. Similarly, $\boldsymbol{\mu_{f,a}}$ is the normalized Haar measure on $G_{f, a}$.
    \item $\boldsymbol{\bar{\sigma}}$ denotes the smallest closed and normal subgroup containing a given $\sigma\in\text{Gal}(M/L)$ for a Galois extension $M/L$.
    \item A \textbf{minimal Galois field extension} $M/L$ is a nontrivial Galois extension with no nontrivial, proper Galois subextension.
    \item $\boldsymbol{\zeta(d^n)}$ is the set of all $d^n$-th roots of unity. Similarly, $\boldsymbol{\zeta(d^\infty)}$ is the set of $d^n$-th roots of unity for all $n$. Also, $\boldsymbol{\zeta^\infty}$ is the set of all roots of unity.
    \item $\boldsymbol{L_d}$ is $L(\zeta(d^\infty))$. 
\end{itemize}
\end{definition}
\begin{definition}\label{open max normal def}
$N\triangleleft G_{f}$ is an \textbf{open maximal normal subgroup} when it is normal, open in the Krull topology (or equivalently closed and of finite index in $G_{f}$), and is maximal among all open normal subgroups (i.e. there is no open normal subgroup strictly between $N$ and $G_{f}$). A similar definition applies for $N\triangleleft G_{f, a}$.
\end{definition}
Now we can give our result.
\begin{theorem}[Main Result]\label{main theorem} Let $k$ be a countable Hilbertian field (of any characteristic). Given $f\in k^{sep}[x]$, let $k_f/k$ be the smallest extension of $k$ over which $f(x)$ is defined. Let $S$ be the set of $\sigma\in G_k$ where $F^{\bar{\sigma}}$ is a pseudo-algebraically closed Hilbertian field, over which any polynomial $f(x)$ of degree $d>1$ not divisible by $\text{char}(k)$ with $G_f$ over $k_f$ having infinitely many open maximal normal subgroups satisfies $[G_f: G_{f,a}]=\infty$ over $F^{\bar{\sigma}}$ for all $a\in F^{\bar{\sigma}}$.
We have that $\mu_k(S)=1$.
\end{theorem}
It will be shown that $G_f$ over $k_f$ having infinitely many open maximal normal subgroups implies the same condition for $G_f$ over $F^{\bar{\sigma}}$. Hence the stronger condition placed on $G_f$ in the definition of $S$ in Theorem \ref{main theorem} gives us more freedom regarding which fixed fields we are allowed to take. Conversely, it turns out that restricting ourselves to the fixed fields containing all roots of unity allows us to loosen the conditions on the polynomial $f(x)$ to achieve the same infinite index result. Hence we get the following corollary.
\begin{corollary}\label{cor to main theorm roots of unity included}
Let $k$ be a countable Hilbertian field (of any characteristic). Let $S'$ be the set of $\sigma\in G_{k(\zeta^\infty)}$ where $F^{\bar{\sigma}}$ is a pseudo-algebraically closed Hilbertian field, over which any polynomial $f(x)$ of degree $d>1$ not divisible by $\text{char}(k)$ with $G_f$ over $F^{\bar{\sigma}}$ having infinitely many open maximal normal subgroups satisfies $[G_f: G_{f,a}]=\infty$ for all $a\in F^{\bar{\sigma}}$. Then $\mu_{k(\zeta^\infty)}(S')=1$.
\end{corollary}
Clearly any $G_{f,a}$ satisfying the conclusion of Theorem \ref{main theorem} also has infinite index in the larger group $\text{Aut}(T_d^\infty)$. For any polynomial not satisfying the assumption of Theorem \ref{main theorem} (i.e. $G_f$ over $k_f$ has only finitely many open maximal normal subgroups), we trivially have that $[\text{Aut}(T_d^\infty): G_f]=\infty$ over $k_f$ and so over $F^{\bar{\sigma}}$ since $\text{Aut}(T_d^\infty)$ clearly has infinitely many such subgroups. Hence we easily get the following corollary that is a direct generalization of Theorem \ref{DK}.
\begin{corollary}\label{cor to main theorem tree aut gp}
Let $k$ be a countable Hilbertian field (of any characteristic). Let $S''$ be the set of $\sigma\in G_k$ where $F^{\bar{\sigma}}$ is a pseudo-algebraically closed Hilbertian field, over which any polynomial $f(x)$ of degree $d>1$ not divisible by $\text{char}(k)$ with $G_{f,a}$ over $F^{\bar{\sigma}}$ satisfies $[\text{Aut}(T_d^\infty): G_{f,a}]=\infty$ for all $a\in F^{\bar{\sigma}}$. Then $\mu_k(S'')=1$.
\end{corollary}

Given that $[\text{Aut}(T_d^\infty): G_{f, 0}]=\infty$ for every $f\in k[x]$ of degree $d$, Theorem \ref{DK} does not claim that any proper extension $L/k$ satisfies the analogous conclusion. In contrast, Theorem \ref{main theorem} gives infinitely many countable proper extensions $F^{\bar{\sigma}}$ of $k$ satisfying $[G_f:G_{f,a}]=\infty$ for every $a\in F^{\bar{\sigma}}$ and polynomial $f(x)$ over $F^{\bar{\sigma}}$ where $G_f$ has infinitely many open maximal normal subgroups, regardless of whether $k$ satisfies the analogous conclusion.\\

To see this, we use a variation of the argument via quadratic discriminant extensions by Dittmann and Kadets in \cite{DK}. Suppose we have two quadratic (separable) extensions $L_1/k$ and $L_2/k$ such that the natural continuous group homomorphism $\phi:G_k\rightarrow \mathbb{Z}/2\mathbb{Z}\times \mathbb{Z}/2\mathbb{Z}$ is surjective. Given any $\sigma\in G_k$, notice that $N:=\phi^{-1}(\langle \sigma\rangle)$ is a proper open normal subgroup of $G_k$ containing $\sigma$. Hence $\bar{\sigma}\neq G_k$ or equivalently $k\subsetneq F^{\bar{\sigma}}$.\\

In order to obtain such quadratic extensions $k(\sqrt{a})$ and $k(\sqrt{b})$ when $\text{char}(k)\neq 2$, it suffices to choose $a,b\in k$ such that $f(x)=x^2-a$, $g(x)=x^2-b$, and $h(x)=x^2-ab$ are all irreducible over $k$. In the event that $\text{char}(k)=2$, one can choose $a,b\in k$ such that $f(x)=x^2+x+a$,  $g(x)=x^2+x+b$, and $h(x)=x^2+x+a+b$ are all irreducible over $k$. Letting $\alpha$ be a root of $f(x)$ and $\beta$ be a root of $g(x)$, one can check that $\beta\not\in k(\alpha)$, hence giving us the desired quadratic extensions $k(\alpha)/k$ and $k(\beta)/k$.

\subsection{Examples}
Benedetto et. al. in \cite{Tucker} showed that when $G_{f}$ is a pro $p$-group, it has infinitely many open maximal subgroups. Since every open maximal subgroup of any pro $p$-group is also normal, we see that $G_f$ being a pro $p$-group is enough to satisfy the assumptions in Theorem \ref{main theorem}.\\

We say a polynomial is ``post critically finite'' (PCF) when its critical points are preperiodic. Much of these results depend on whether $f(x)$ is PCF. For example, Benedetto et. al. in \cite{Tucker} showed that whenever $f$ is a PCF rational quadratic function, we have $G_f\cong G_{f,a}$ for all $a\in k\setminus \Omega$ where $\Omega$ is a thin set. Assuming the $abc$-conjecture and the conjecture on the irreducibility of iterates in \cite{JonesLevy}, the converse is also true. Moreover, these assumptions imply any non-PCF cubic polynomial has $[G_f:G_{f,a}]=\infty$ (see \cite{Bridyquad}, \cite{Bridyquestion}, \cite{BridyCubic}, \cite{JuulKrieger}). We say $f$ has colliding critical points when the $n$ iterates of two of them coincide for some $n$. In the quadratic rational case, this condition means $f$ is also PCF.\\

Other known examples of polynomials $f$ with $G_f$ having infinitely many open maximal normal subgroups satisfy a related property called ``infinite abelianization''. In order to define it, we must know certain structural properties of PFIMG.
\begin{definition}\label{Def Wreath Prod}
Let $\{G_i\}_{i=1}^\infty$ be a sequence of finite groups with $G_i\leq S_m$ for some fixed $m>1$. The \textbf{wreath product of $\mathbf{G_i}$ and $\mathbf{G_{i+1}}$}, denoted $G_i\wr G_{i+1}$, is
\[G_i \wr G_{i+1}:=G_i\ltimes (\underbrace{G_{i+1}\times G_{i+1}\times \cdots \times G_{i+1}}_{m \ \text{times}}).\]
Iterating this construction over all $i$ gives us the iterated wreath product $\mathbf{[G_i]_{i=1}^\infty}$.
\end{definition}
\begin{proposition}
Given a group $G$, let $G^{ab}$ denote the abelianization of $G$. For an iterated wreath product $G=[G_i]_{i=1}^\infty$, we have $G^{ab}\cong \prod_{i=1}^\infty G_i^{ab}$.
\end{proposition}
Now we can define infinite abelianization.
\begin{definition}[Infinite Abelianization] Let $G=[G_i]_{i=1}^\infty$ be an iterated wreath product. We say $G$ has \textbf{infinite abelianization} when $G_i^{ab}$ is nontrivial for infinitely many $i$.
\end{definition}
Notice if $G_f$ is isomorphic to an iterated wreath product with infinite abelianization, then it trivially has infinitely many open maximal normal subgroups. Of course the case considered by Dittmann and Kadets in \cite{DK} where $\prod_{i=1}^\infty \mathbb{Z}/2\mathbb{Z}\subset G_{f,0}^{ab}$ is one where we have infinite abelianization.\\

It is straightforward to verify that $f(x)=x^d-a$ for any $a$ in the ground field such that $f$ is not PCF has $G_f$ with infinite abelianization (over fields containing a primitive $d$-th root of unity). In particular, when $d=2$ this ensures that $G_f$ satisfies the condition described in Theorem \ref{main theorem} over any countable Hilbertian field.\\

Conversely, when $a=0$, making $f$ PCF, one can check that $G_f$ is topologically cyclic, meaning it has only one open maximal subgroup. It turns out that the conclusion of Theorem \ref{main theorem} is also false in this case. In fact, the index will be exactly 1 over any $a$ not a $d$-th power in the ground field (such an $a$ always exists in any Hilbertian field).\\

Benedetto et. al. describe the structure of the $G_f$ resulting from a a non-PCF quadratic rational function in \cite{BenedettoRational} and a non-PCF cubic polynomial in \cite{BenedettoCubic}. These $G_f$ are easily seen to have infinite abelianization. Notice Theorem \ref{main theorem} implies that $[G_f:G_{f,a}]=\infty$ over certain Hilbertian fields without assuming the $abc$-conjecture or the conjecture on irreducibility of iterates in \cite{JonesLevy}.

\section{Preliminaries}
We begin by stating a key fact by Fried and Jarden in 2023 (Theorem 21.12.2 in \cite{FJ}).
\begin{theorem}[Fried-Jarden, \cite{FJ}]\label{almost every field is PAC} Let $k$ be a countable Hilbertian field and $A$ be the set of $\sigma\in G_k$ such that $F^{\bar{\sigma}}$ is pseudo-algebraically closed. Then $\mu_k(A)=1$.
\end{theorem}
It should be emphasized that the analogous claim of Theorem \ref{almost every field is PAC} over an arbitrary Galois extension is false (see the introduction to Chapter 31 in \cite{FJ}). In contrast, Bary-Soroker and Fehm's 2013 result does not rely on this distinction.
\begin{theorem}[Bary-Soroker and Fehm, \cite{Bary}]\label{Bary}  Let $k$ be a countable Hilbertian field, $L/k$ Galois, and $B$ be the set of $\sigma\in \text{Gal}(L/k)$ such that $F^{\bar{\sigma}}$ is Hilbertian. Then $\mu_k(B)=1$.
\end{theorem}

We make use of Kuyk's Corollary of Haran's Diamond Theorem in 1970.
\begin{theorem}\label{Abelian Extensions} (Kuyk, \cite{AbelianExt}) Let $k$ be a Hilbertian field and $L/k$ an Abelian extension. Then $L$ is Hilbertian. 
\end{theorem}
Since all cyclotomic extensions are Abelian, we can replace any given Hilbertian field $k$ with $k$ without losing the Hilbertian property by Theorem \ref{Abelian Extensions}. Note $k$ is countable whenever $k$ is such. Letting $L$ be the separable closure of $k$ in Theorem \ref{Bary}, we obtain part of Theorem \ref{main theorem} as a Corollary of Theorems \ref{Bary} and \ref{Abelian Extensions}.\\

Next we state a very useful result by Michael Fried in 1973.
\begin{theorem}\label{Roots of unity}(Fried, \cite{Fried}) Given a field $k$ of characteristic $p\geq 0$ and $f(x)/k$ a polynomial of degree $d$ such that $p\nmid d$, we have
\[f^{-1}(t)\cap \bar{k}\subset k(\zeta(d)).\]
\end{theorem}
One easily obtains the following.
\begin{corollary}[Corollary to Theorem \ref{Roots of unity}]\label{cor roots of unity}
Let $k$ be a field of characteristic $p\geq 0$ and $f(x)/k$ a polynomial of degree $d$ such that $p\nmid d$. Suppose $k=k(\zeta(d^\infty))$. Then 
\[{f^{-\infty}(t)\cap \overline{k}\subset k}.\]
\end{corollary}
Applying this to Definition \ref{Galois group notation} gives us that $G_f$ over $k=k(\zeta(d^\infty))$ is isomorphic to $G_f$ over $\bar{k}$ (as no preimage of $t$ under $f(x)$ appears in the fixed field $\bar{k}(t)$ of the latter, unless it also appears in the fixed field $k(t)$ of the former).\\

This gives us a simple yet successful strategy for proving  Theorem \ref{main theorem}  for a given polynomial $f(x)/F^{\bar{\sigma}}$ satisfying the relevant assumptions whenever $k=k(\zeta^\infty)$. Let $a\in F^{\bar{\sigma}}$ and $L/k$ be a finite subextension of $F^{\bar{\sigma}}/k$ containing $a$ and the coefficients of $f(x)$. By Corollary \ref{cor roots of unity}, we know $G_f$ remains unchanged when working over either $L$ or $F^{\bar{\sigma}}$. By basic Galois theory, we know $G_{f,a}$ over $F^{\bar{\sigma}}$ can be identified with a subgroup of $G_{f,a}$ over $L$. Hence $[G_f: G_{f,a}]$ over $L$ is smaller than $[G_f: G_{f,a}]$ over $F^{\bar{\sigma}}$.
Clearly if the smaller index over $L$ is already infinite, then we are done. If not, the idea behind the proof is constructing $F^{\bar{\sigma}}$ so that $G_{f,a}$ over $F^{\bar{\sigma}}$ is ``much smaller'' than $G_{f, a}$ over $L$. This will ensure that the index over $F^{\bar{\sigma}}$ is ``infinitely larger'' than the index over $L$.\\

More precisely, given a field $F$, we use a form of the Borel-Cantelli Theorem to deduce that almost surely a randomly chosen $\sigma\in G_F$ will give a fixed field $F^{\bar{\sigma}}$ that contains infinitely many of a given sequence of minimal Galois extensions of $F$. This can be generalized to a countable list of finite extensions $F_i/k(\zeta^\infty)$ to state that with $\mu_{k}$-probability 1, for each $F_i$ we have either $F_i\not\subset F^{\bar{\sigma}}$ or $F^{\bar{\sigma}}$ contains infinitely many minimal Galois extensions of $F_i$. We apply this result to the sequence of all such finite extensions of $k$ with minimal Galois extensions arising from polynomials having $G_f$ satisfying the conditions in Theorem \ref{main theorem} and $[G_f: G_{f,a}]<\infty$ over $F_i$ for some $a\in F_i$. The Galois correspondence then easily shows that $[G_f:G_{f,a}]=\infty$ over $F^{\bar{\sigma}}$, finishing the proof. The proof when working over an arbitrary countable Hilbertian field follows essentially the same strategy, except that we require extra propositions to deal with the complications arising from changes in $G_f$ when adding or removing roots of unity to the given field.

\section{Derivation of Our Theorem}
The essence of the proof of Theorem \ref{main theorem} lies in a probabilistic result regarding minimal Galois extensions of Hilbertian fields contained in certain fixed fields. In order to derive this result, we will need a generalized form of the Borel-Cantelli Theorem. Before stating it, we provide important background on how Haar measure relates to linearly disjoint field extensions. First recall that given a sequence of distinct measurable subsets $\{N_j\}_{j=1}^\infty$ of a probability space $G$, we say the $N_j$ are pairwise $\mu$-independent when $\mu(N_j\cap N_{\ell})=\mu(N_j)\mu(N_{\ell})$ for all $1\leq j< \ell <\infty$. Similarly, we say the $N_\ell$ are mutually $\mu$-independent when $\mu(\bigcap_{j\in J}N_j)=\prod_{j\in J} \mu(N_j)$ for any finite subset of indices $J$. It should be realized that pairwise $\mu$-independence is generally a weaker condition than mutual $\mu$-independence. The definitions of being $\mu$-independent and linearly disjoint are very closely connected by the Galois correspondence when $G$ is a Galois group.
\begin{lemma}[Being $\mu$-Independent and Linearly Disjoint]\label{indep=dis} 
Consider two  open subgroups $H$ and $J$ of $G_L$ for some field $L$. Denote the corresponding finite field extensions by $L_H$ and $L_J$. $H$ and $J$ are $\mu$-independent if and only if the $L_H/L$ and $L_J/L$ field extensions are  linearly disjoint.  
\end{lemma}
\begin{proof}
By the definition of Haar measure, 
\begin{equation}\label{eqmsureofint}
\mu_L(H\cap J)=[L_HL_J:L]^{-1}
\end{equation}
and
\begin{equation}\label{eqprodofmsure}
\mu(H)\mu(J)=[L_H:L]^{-1}[L_J:L]^{-1}.
\end{equation}
Clearly the right sides of Equations \eqref{eqmsureofint} and \eqref{eqprodofmsure} are equal precisely when the field extensions $L_H/L$ and $L_J/L$ are linearly disjoint. 
\end{proof}
With Lemma \ref{indep=dis} in mind, we can now appreciate the following example.
\begin{example}[Pairwise vs Mutual $\mu$-Independence]\label{exindep}
For mutually distinct primes $p$, $q$, and $s$, consider the field extensions $\mathbb{Q}(\sqrt{pq})/\mathbb{Q}$, $\mathbb{Q}(\sqrt{ps})/\mathbb{Q}$, and $\mathbb{Q}(\sqrt{q/s})/\mathbb{Q}$. Note this sequence of field extensions is pairwise linearly disjoint  and each of these extensions has no non-trivial subextensions due to their prime degree). This sequence is not mutually linearly disjoint, however, because
\[[\mathbb{Q}(\sqrt{pq}, \sqrt[r]{ps}, \sqrt{q/s}):\mathbb{Q}]=r^{2}\neq r^{3}=[\mathbb{Q}(\sqrt{pq}):\mathbb{Q}][\mathbb{Q}(\sqrt{ps}):\mathbb{Q}][\mathbb{Q}(\sqrt{q/s}):\mathbb{Q}].\]
Letting $H_{pq}, H_{ps}, H_{qs}$ be the corresponding subgroups in $G_{\mathbb{Q}}$, we see that these subgroups are pairwise $\mu_{\mathbb{Q}}$-independent but not mutually $\mu_{\mathbb{Q}}$-independent by Lemma \ref{indep=dis}.
\end{example}
The classical form of the Borel-Cantelli Theorem requires a collection of mutually $\mu$-independent subsets. We instead give a stronger form of the Borel-Cantelli Theorem that applies the same conclusion to a collection of merely pairwise $\mu$-independent subsets. This version is often considered to be a corollary to the Kochen and Stone Inequality in \cite{BCOriginal}. For an explicit statement and proof of the version given below, one should consult Lemma 1.1 in \cite{BCPairwiseIndependent}.
\begin{theorem}[Borel-Cantelli]\label{BC} Let $\{N_{j}\}_{j=1}^\infty$ be a sequence of measurable subsets of a probability space $G$ that are pairwise $\mu$-independent and satisfy $\sum_{j=1}^\infty \mu(N_{j})=\infty$. Put
\[S:=\bigcap_{n=1}^\infty \bigcup_{j=n}^\infty N_{j}.\]
Note $S$ is the set of elements of $G$ that belong to infinitely many $N_{j}$. We have $\mu(S)=1$.
\end{theorem}
Lemma 31.1.2 in \cite{FJ} states that every countably infinite sequence of pairwise linearly disjoint Galois extensions contains an infinite subsequence of mutually linearly disjoint extensions. Hence one might observe that the proofs of Theorem \ref{main theorem} and the prerequisite lemmas can be justified using only the classical pairwise independent form of the Borel-Cantelli Theorem. We use the stronger Theorem \ref{BC} merely for notational convenience in avoiding passing to such subsequences in various steps.\\

Theorem \ref{BC} gives us a fairly simple lemma. One might note that it is a special case of Lemma 21.5.2 in \cite{FJ}, where each set of left cosets in that theorem is precisely the identity coset in our result.
\begin{lemma}\label{Lemma First}
Let $F$ be a countable Hilbertian field. Suppose we have a constant $C$ and a countably infinite set of distinct minimal Galois extensions $\{L_j/F\}_{j=1}^\infty$ such that $[L_j:F]\leq C$ for all $j$. Define $T$ to be the set of all $\sigma\in G_F$ such that $F^{\bar{\sigma}}$ is a pseudo-algebraically closed Hilbertian field containing infinitely many of the $L_j$ fields. Then $\mu_F(T)=1$. 
\end{lemma}
\begin{proof}[Proof of Lemma \ref{Lemma First}]
Let $N_j\triangleleft G_F$ be the open maximal normal subgroup corresponding to the minimal Galois extension $L_j/F$. We have
\begin{equation}\label{eq Lemma First}
\mu_F(N_j)=\frac{1}{[G_F: N_j]}=\frac{1}{[L_j:F]}\geq \frac{1}{C}.
\end{equation}
It is known that distinct minimal Galois extensions are always pairwise linearly disjoint. Hence the collection of $\{N_j\}_{j=1}^\infty$ are pairwise $\mu_F$-independent by Lemma \ref{indep=dis}. The constant, positive lower bound in Equation \eqref{eq Lemma First} implies that $\sum_{j=1}^\infty \mu_F(N_j)=\infty$. By Theorem \ref{BC}, we have that the set of elements $T'$ of $G_F$ belonging to infinitely many $N_j$ has full measure.\\

By the Galois correspondence, $L_j\subset F^{\bar{\sigma}}$ if and only if $\bar{\sigma}\leq N_j$. In fact, it is enough to require $\sigma\in N_j$ by the definition of $\bar{\sigma}$. Hence $\sigma\in T'$ is equivalent to the condition that $F^{\bar{\sigma}}$ contains infinitely many of the $L_j$ fields.\\

Recall a countable union of measure zero sets in any measure space always has measure zero (by the countable subadditivity property). As the complement of a countable intersection of full measure sets is such a union, it follows that these intersections necessarily retain full measure. Hence we form $T$ by intersecting $T'$ with the full measure subsets of $G_F$ from Theorems \ref{almost every field is PAC} and \ref{Bary} to finish the proof.
\end{proof}
Next we generalize Lemma \ref{Lemma First} to a finite extension of a Hilbertian field. This is done using similar techniques to those in the proof of Theorem \ref{Bary}, as written in the proof of Theorem 31.3.2 in \cite{FJ}.
\begin{lemma}\label{Lemma second}
Let $k$ be a countable Hilbertian field and $F/k$ a field extension. Let $L_j/F$ be as in Lemma \ref{Lemma First}. Define $W$ to be the set of all $\sigma\in G_k$ where $F^{\bar{\sigma}}$ is a pseudo-algebraically closed Hilbertian field that either does not contain $F$ or contains infinitely many of the $L_j$ fields. Then $\mu_k(W)=1$. 
\end{lemma}
\begin{proof}[Proof of Lemma \ref{Lemma second}] Define $T$ as in Lemma \ref{Lemma First} with respect to the field $F$, which implies $\mu_F(T)=1$. Letting $T^c:=G_F\setminus T$, define
\begin{equation}
    W':=T\sqcup (G_k\setminus G_F)=G_k\setminus T^c.
\end{equation}
By identifying $G_F$ with a subgroup of $G_k$, we can view $W'\subset G_k$. Note $\mu_k(T^c)\leq \mu_F(T^c)=0$. Hence $\mu_k(W')=1$. Given $\sigma\in W'$, suppose $F\subset F^{\bar{\sigma}}$. This implies
\begin{equation}
\sigma\in G_F\cap W'=T.
\end{equation}
By the definition of $T$, we know $F^{\bar{\sigma}}$ is a pseudo-algebraically closed Hilbertian field containing infinitely many $L_j$. In order to ensure that $F^{\bar{\sigma}}$ is still a pseudo-algebraically closed Hilbertian field in the case that $F\not\subset F^{\bar{\sigma}}$, we simply form $W''$ by intersecting $W'$ with the full measure subsets of $G_k$ from Theorems \ref{almost every field is PAC} and \ref{Bary}, just as in the proof of Lemma \ref{Lemma First}. As $W''= W$ and Haar measure is complete, this finishes the proof.
\end{proof}
Notice that when $F/k$ is infinite (not necessarily Galois), the Galois correspondence implies that $F^{\bar{\sigma}}$ simply does not contain $F$ for almost every $\sigma$ in $G_k$. However, we will only be concerned with the case that $[F:k]<\infty$ in the proof of Theorem \ref{main theorem}. In this setting, there is a positive probability in $G_k$ that $F^{\bar{\sigma}}$ contains $F$ (by passing to the Galois closure of $F/k$ and using the correspondence described in the proof of Lemma \ref{Lemma First}), meaning the claim regarding the containment of infinitely many $L_j$ fields is indeed non-trivial.\\

Now we easily generalize Lemma \ref{Lemma second} over countably many finite extensions of a Hilbertian field $k$ by basic measure theoretic properties.
\begin{proposition}\label{Prop for main thm}
Let $k$ be a countable Hilbertian field and $\{F_i/k\}_{i=1}^\infty$ be a countable set of extensions of $k$. For each $i$, suppose we have a constant $C_i$ and a countably infinite set of distinct minimal Galois extensions $\{L_{i,j}/F_i\}_{j=1}^\infty$ such that $[L_{i,j}:F_i]\leq C_i$ for all $j$. Define $S$ to be the set of all $\sigma\in G_k$ where $F^{\bar{\sigma}}$ is a pseudo-algebraically closed Hilbertian field that for each $i$ satisfies exactly one of the following,
\begin{itemize}
    \item $F_i\not\subset F^{\bar{\sigma}}$
    \item $L_{i,j}\subset F^{\bar{\sigma}}$ for infinitely many $j$ values.
\end{itemize}
We have $\mu_k(S)=1$.
\end{proposition}
\begin{proof}
For each $i$, we can apply Lemma \ref{Lemma second} to $F_i/k$ and $\{L_{i,j}/F_i\}_{j=1}^\infty$. This gives us full measure sets $W_i\subset G_k$ such that for each $\sigma\in W_i$, we have that $F^{\bar{\sigma}}$ is a pseudo-algebraically closed Hilbertian field that either does not contain $F_i$ or contains infinitely many of the $L_{i,j}$ fields. Note $S=\bigcap_{i=1}^\infty W_i$. As each $W_i$ has full measure, so does $S$ (see the end of the proof of Lemma \ref{Lemma First}).
\end{proof}
Essentially Theorem \ref{main theorem} is proven by using the correspondence between open maximal normal subgroups and minimal Galois field extensions and then applying Proposition \ref{Prop for main thm}. In order to satisfy the index bound condition in Proposition \ref{Prop for main thm}, we utilize a very useful restriction resulting from the classification of finite simple groups due to Babai, Goodman, and Pyber in 1997.
\begin{theorem}[Babai-Goodman-Pyber, \cite{SimpleGp}]\label{Finite Simple Gp Bd} Given $r\in \mathbb{N}$, let $G$ be a finite simple group where $|G|$ has no prime divisor greater than $r$. Then $|G|<r^{r^2}$.
\end{theorem}
Given a countable Hilbertian field $k$, $a\in k$,  $f\in k[x]$ with degree not divisible by $\text{char}(k)$, and an algebraic extension $L/k$, Theorem \ref{Roots of unity} indicates that $G_f$ over $L$ may be smaller than $G_f$ over $k$, unless $k$ already contains all $d^n$-th roots of unity for every $n$. Indeed, one could now prove Theorem \ref{main theorem} using only what has been discussed so far, with the additional assumption that $k$ contains all roots of unity. It is only to deal with the complications that arise from $G_f$ changing in size over extensions of an arbitrary countable Hilbertian field $k$ that we need the following propositions.\\

Of course, $G_{f,a}$ over $L$ can also be smaller than $G_{f,a}$ over $k$ (regardless of whether $k$ contains all roots of unity). The next proposition we obtain provides crucial information on how $[G_f: G_{f,a}]$ changes as we alter the base field. First we require three lemmas.
\begin{lemma}\label{lemma gen preimages and sep alg preimages}
Let $L$ be an infinite field, $a\in L$, $f\in L[x]$ of degree $d>1$, and $f^n(x)-a$ be separable. Then we have $k(f^{-n}(t))\cap \bar{k}\subset k(f^{-n}(a))$.
\end{lemma}
\begin{proof}
Let $K=k(f^{-n}(t))\cap \bar{k}$ and $A=k[f^{-n}(t)]$, which is clearly integral over $k[t]$. Moreover, $A$ is the set of all elements in $k(f^{-n}(t))$ that are integral over $k[t]$. In particular, $K\subset A$.\\

Let $P_n$ be a prime ideal in $A$ lying above the prime ideal $(t-a)$ in $k[t]$. As $f^n(x)-a$ is separable, the associated inertia subgroup $I_n$ is trivial. Letting $L_n=A/P_n$, by Equation \eqref{Decomp Subgroup n levels} we have a $L$-linear isomorphism $\theta: L_n\rightarrow L(f^{-n}(a))$. Lifting back to $A$ gives up a natural $L$-linear ring homomorphism $\phi: A\rightarrow k(f^{-n}(a))$.\\

Since $K\subset A$, we can consider the restriction of $\phi$ to $K$. As $K\cong K/P_n$, we have that $\phi$ is a $L$-linear embedding of $K$ into $L(f^{-n}(a))$. Since $K/L$ is Galois, we have that $\phi(K)=K$. This finishes the proof.
\end{proof}
The ``post critical set'' of a given polynomial $f(x)$ is the set of critical points and all their forward iterated images.
\begin{lemma}\label{lemma post-critical and sep} Let $L$ be a field, $a\in L$, and $f\in L[x]$ of degree $d>1$ such that $\text{char}(L)\nmid \deg (f)$. If $a$ is not in the post-critical set then $f^n(x)-a$ is separable for all positive integers $n$.
\end{lemma}
\begin{proof}
We prove the contrapositive, so suppose $f^n(x)-a$ is inseparable for some $n$. Let $r$ denote the repeated root. As $\text{char}(L)\nmid \text{deg}(f)$, we must have that $(f^n)'(r)=0$. As the chain rule implies $(f^n)'(x)=\prod_{i=0}^{n-1} f'(f^i(x))$ (where $f^0(x):=x$), we have $f'(f^m(r))=0$ for some $0\leq m\leq n-1$. By definition, $\alpha=f^m(r)$ is a critical point of $f^n(x)$. Notice that $f^{n-m}(\alpha)-a=f^n(r)-a=0$, finishing the proof.
\end{proof}
The converse of Lemma \ref{lemma post-critical and sep} is also true but we omit the proof since we do not require this direction.

\begin{lemma}\label{lemma gen preimages and all alg preimages} Let $L$ be an infinite field, $a\in L$ and $f\in L[x]$ of degree $d>1$ such that $\text{char}(L)\nmid \deg (f)$ and $|f^{-\infty}(a)|=\infty$. Then $\exists$ a non-negative integer $m$ depending only on $f(x)$ and $a$ such that $\forall \ n\geq 1$, we have $L(f^{-n}(t))\cap \bar{L}\subset L(f^{-(n+m)}(a))$.
\end{lemma}
\begin{proof}
If $a$ is not in the post-critical set, by Lemmas \ref{lemma gen preimages and sep alg preimages} and \ref{lemma post-critical and sep} we have the desired claim with $m=0$. Otherwise $a$ is contained in the post-critical set. Let $\{\alpha_i\}_{i=1}^j$ be all the critical points of $f(x)$ such that for each $i$ we have $f^{n_i}(\alpha_i)=a$ for some non-negative integer $n_i$ (note $a$ itself could be a critical point). Observe that if $\alpha_i$ is wandering (i.e. not preperiodic), then 
\begin{equation}\label{eq wandering trivial intersection}
\left(\bigcup_{n>n_i}f^n(\alpha_i)\right)\cap (\{a\}\cup f^{-\infty}(a))=\emptyset.    
\end{equation}
Equation \eqref{eq wandering trivial intersection} immediately gives that
\begin{equation}\label{eq critical pt orbit cap a preimages}
\left|\left(\bigcup_{n=0}^\infty f^n(\alpha_i)\right)\cap (\{a\}\cup f^{-\infty}(a))\right|<\infty.
\end{equation}
In the event that $\alpha_i$ is preperiodic, the fact that $|\bigcup_{n=0}^\infty f^n(\alpha_i)|<\infty$ implies Equation \ref{eq critical pt orbit cap a preimages} still holds. Either way, there is a positive integer $m_i$ such that for any $\beta_i\in f^{-s}(a)$ and $s\geq m_i$, we have that $\beta_i\neq f^n(\alpha_i)$ for all non-negative integers $n$. Let $m:=\max_{1\leq i\leq j} (m_i)$. Then any $\beta\in f^{-m}(a)$ is not in the post-critical set. Again by Lemmas \ref{lemma gen preimages and sep alg preimages} and \ref{lemma post-critical and sep}, it follows that
\begin{equation}
L(f^{-n}(t))\cap \bar{L}\subset L(f^{-n}(\beta))\subset L(f^{-(n+m)}(a)).
\end{equation}

\end{proof}

\begin{proposition}\label{prop group index bound} Let $L$ be an infinite field, $M/L$ an algebraic extension, $a\in L$, and $f\in L[x]$ of degree $d>1$  such that $\text{char}(L)\nmid \deg (f)$. Then $[G_f: G_{f,a}]$ over $M$ is bounded from below by $c[G_f: G_{f,a}]$ over $L$ for some positive constant $c$ depending only on $a$ and $f(x)$.
\end{proposition}
\begin{proof}
Let $G_f$ and $G_{f,a}$  denote the generic and specialized PFIMG over $L$. We write $H_f$ and $H_{f,a}$ for the generic and specialized PFIMG over $M$, since they can be identified with subgroups of $G_f$ and $G_{f,a}$. First notice that if $|f^{-\infty}(a)|<\infty$ then $[G_f:G_{f,a}]=[H_f: H_{f,a}]=\infty$ trivially. Hence we can assume $|f^{-\infty}(a)|=\infty$. Denote $M_{t,n}:=M\cap L(f^{-n}(t))$ and $M_{a,n}:=M\cap L(f^{-n}(a))$. Corollary \ref{cor roots of unity} and Lemma \ref{lemma gen preimages and all alg preimages} imply that for some $m\geq 1$ and all $n\geq m$, we have $M_{n,t}M_{a,n}= M_{a,n}(\zeta(d^{\ell}))$ for some $\ell\leq n+m$. We also know $\zeta(d^{\ell})\in M_{a, n+m}$. Since $m$ is independent of $n$, we can write
\begin{equation}\label{eq gp index ratio bd}
\frac{[G_f: G_{f,a}]}{[H_f:H_{f,a}]}=\frac{[M\cap L(f^{-n}(t)):L]}{[M\cap L(f^{-n}(a)):L]}\leq d^m.
\end{equation}
Note the equality comes from Definition \ref{Group Index Def} and the Galois fixed field correspondence. Observe that the leftmost and rightmost sides of Equation \eqref{eq gp index ratio bd} give us 
\begin{equation}\label{eq gp index lower bd}
\frac{1}{d^m}[G_f: G_{f,a}]\leq [H_f: H_{f,a}].
\end{equation}
Clearly Equation \eqref{eq gp index lower bd} finishes the proof, letting $c=d^{-m}$.
\end{proof}
Considering Corollary \ref{cor roots of unity} and Proposition \ref{Prop for main thm}, it is unsurprising that we need information on how minimal Galois extensions over a base field $L$ lift as we adjoin roots of unity. In order to obtain the desired proposition, we analyze properties of $\text{Gal}(L_d/L)$. First we give a description of the structure of the unit group  $\mathbb{Z}^\star_p$. In order to do so, we require certain identifications due to Jurgen Neukirch in \cite{p-adicunitgp}.
\begin{lemma}[Proposition 5.5 in \cite{p-adicunitgp}]\label{lemma unit gp iso to subgps of Zp} Let $K|\mathbb{Q}_p$ be a $\mathfrak{p}$-adic number field with valuation ring $\mathcal{O}$ and maximal ideal $\mathfrak{p}$, and let $p\mathcal{O}=\mathfrak{p}^e$. Let $U^{(n)}=1+\mathfrak{p}^n$. Then the power series for $\exp (x)$ and $\log (1+xz)$ yield, for $n>e/(p-1)$, two mutually inverse isomorphisms and homeomorphisms 
$\begin{tikzcd}
  \mathfrak{p}^n \arrow[r, shift left, "\exp"] & U^{(n)} \arrow[l, shift left, "\log"]
\end{tikzcd}$.
\end{lemma}
Over $\mathbb{Q}_p$ we have $\mathfrak{p}=p\mathbb{Z}_p$  and $\mathcal{O}=\mathbb{Z}_p$ in Lemma \ref{lemma unit gp iso to subgps of Zp}, meaning $e=1$. Hence $U^{(n)}\cong \mathfrak{p}^n$ as groups for all $n>(p-1)^{-1}$. In particular, Lemma \ref{lemma unit gp iso to subgps of Zp} implies that $U^{(1)}\cong \mathfrak{p}$ whenever $p$ is an odd prime.
\begin{lemma}\label{lemma p-adic units} Let $p$ be prime. If $p$ is odd, we have
\begin{equation}\label{eq p-adic unit semidirect prod odd p}
\mathbb{Z}_p^\star \cong \mathbb{Z}/(p-1)\mathbb{Z}\times \mathbb{Z}_p.
\end{equation}
If $p=2$, we have
\begin{equation}\label{eq p-adic unit semidirect prod p=2}
\mathbb{Z}_2^\star \cong \mathbb{Z}/2\mathbb{Z\times \mathbb{Z}}_2.
\end{equation}
\end{lemma}
\begin{proof}
First we prove the case where $p$ is odd. Defining $U^{(n)}$ as in Lemma \ref{lemma unit gp iso to subgps of Zp}, observe that we have the following short exact sequence (using the natural inclusion and projection maps).
\begin{equation}\label{eq short exact seq odd p}
    0\rightarrow U^{(1)}\rightarrow \mathbb{Z}_p^\star \rightarrow (\mathbb{Z}/p\mathbb{Z})^\star\rightarrow 0
\end{equation}
Let $i$ be the inclusion map of $(\mathbb{Z}/p\mathbb{Z})^\star$ into $\mathbb{Z}_p^\star$ (which is well defined by Hensel's Lemma). If we denote the projection map from $\mathbb{Z}_p^\star $ to $(\mathbb{Z}/p\mathbb{Z})^\star$ by $f$, notice that $f\circ i$ is the identity map on $(\mathbb{Z}/p\mathbb{Z})^\star$. Hence the short exact sequence in Equation \eqref{eq short exact seq odd p} splits, giving us $\mathbb{Z}_p^\star\cong (\mathbb{Z}/p\mathbb{Z})^\star\ltimes U^{(1)}$. Since both factor groups are contained in the abelian group $\mathbb{Z}_p^\star$, the semidirect product reduces to a standard direct product. Clearly $(\mathbb{Z}/p\mathbb{Z})^\star\cong \mathbb{Z}/(p-1)\mathbb{Z}$. By Lemma \ref{lemma unit gp iso to subgps of Zp}, we get Equation \eqref{eq p-adic unit semidirect prod odd p}.\\

When $p=2$, the argument is essentially the same. We start with a short exact sequence similar to that in Equation \eqref{eq short exact seq odd p}.
\begin{equation}\label{eq short exact seq p=2}
0\rightarrow U^{(2)}\rightarrow \mathbb{Z}_2^\star \rightarrow (\mathbb{Z}/4\mathbb{Z})^\star\rightarrow 0
\end{equation}
Letting $j$ be the lift of $(\mathbb{Z}/4\mathbb{Z})^\star$ into $\mathbb{Z}_2^\star$ and $g$ the projection map of $\mathbb{Z}_2^\star$ to $(\mathbb{Z}/4\mathbb{Z})^\star$, again $g\circ j$ is the identity map. Hence the short exact sequence in Equation \eqref{eq short exact seq p=2} splits, giving us $\mathbb{Z}_2^\star\cong (\mathbb{Z}/4\mathbb{Z})^\star\ltimes  U^{(2)}$. Again, the semidirect product reduces to a regular direct product. Clearly $(\mathbb{Z}/4\mathbb{Z})^\star\cong \mathbb{Z}/2\mathbb{Z}$. Lemma \ref{lemma unit gp iso to subgps of Zp} gives us Equation \eqref{eq p-adic unit semidirect prod p=2}, thus establishing the result.
\end{proof}
\begin{lemma}\label{lemma cyclotomic extensions finite index}
Given a field $L$ and $d>1$ such that $\text{char}(L)\nmid d$ and $L_d/L$ is infinite, let $A=\{p_i\}_{i=1}^j$ be the subset of the collection of distinct prime divisors of $d$ such that for each $i$ we have $L(\zeta({p_i}^\infty))\not\subset  L(\zeta(p_1^\infty), \cdots , \zeta(p_{i-1}^\infty))$. If $A\neq \emptyset$, we have an embedding $\phi:\text{Gal}(L_d/L)\xhookrightarrow{} \prod_{i=1}^j \mathbb{Z}_{p_i}^\star$ such that the image of $\text{Gal}(L_d/L)$ has finite index in $\prod_{i=1}^j \mathbb{Z}_{p_i}^\star$.
\end{lemma}
\begin{proof}
Let $L_1:=L$ and $L_i:= L(\zeta({p_1^\infty})), \cdots , \zeta({p_{i-1}^\infty}))$ for $2\leq i\leq j$. Notice $\text{Gal}(L_d/L)\xhookrightarrow{} \prod_{i=1}^j \Gal (L(\zeta(p_i^\infty))/L)$. Let $L_\zeta:=\prod_{i=1}^j L(\zeta(p_i))$. Then $\Gal (L(\zeta(p_i^\infty))/L_\zeta)\xhookrightarrow{}  1+p_i\mathbb{Z}_{p_i}^\star$ for odd $p_i$ or $1+4\mathbb{Z}_2^\star$ when $p_i=2$. This implies all subextensions of $L(\zeta(p_i^\infty))/L_\zeta$ are cyclotomic. Hence for each $i$ we have $L(\zeta({p_i^\infty}))\cap L_\zeta L_i=L_\zeta (\zeta({p_i^{n_i}}))$ for some positive integer $n_i$. Letting $m_i:=[L_\zeta (\zeta(p_i^{n_i})):L_\zeta]$, we see the image of $\text{Gal}(L_d/L_\zeta)$ has index $\prod_{i=1}^j m_i$ in $\prod_{i=1}^j \Gal (L(\zeta(p_i^\infty))/L)$.\\

Moreover, $\Gal (L(\zeta(p_i^\infty))/L)\xhookrightarrow{} \mathbb{Z}_{p_i}^\star$. By assumption, $L_\zeta$ contains only $\zeta(p_i^{\ell_i})$ for some $1\leq \ell_i\leq n_i$. The image of $\Gal (L_\zeta(\zeta(p_i^\infty))/L_\zeta)$ is isomorphic to $1+p_i^{\ell_i}\mathbb{Z}_{p_i}$ for odd $p_i$ or $1+2^{\ell_i+1}\mathbb{Z}_{2}$ when $p_i=2$. This has index $(p-1)p_i^{\ell_i-1}$ (or $2^{\ell_i+1}$ when $p_i=2$) in $\mathbb{Z}_{p_i}^\star$. As $\Gal (L_\zeta(\zeta(p_i^\infty))/L_\zeta)$ can be identified with a subgroup of $\Gal (L(\zeta(p_i^\infty))/L)$, we have that the latter also has image of finite index in $\mathbb{Z}_{p_i}^\star$. This in turn  shows the image of $\text{Gal}(L_d/L_\zeta)$ has finite index in $\prod_{i=1}^j \Gal (L(\zeta(p_i^\infty))/L)$. As $\text{Gal}(L_d/L_\zeta)$ is a subgroup of $\text{Gal}(L_d/L)$, the claim follows.
\end{proof}

Now we require only a basic fact about topologically finitely generated groups in order to later obtain a key proposition.
\begin{lemma}\label{lemma fin gen=fin many open subgp}
Let $m$ be a positive integer. If $G$ is a topologically finitely generated group, then $G$ has only finitely many open subgroups $H$ such that $[G:H]\leq m$.
\end{lemma}
\begin{proof}
Fix the value of $r\leq m$. Let $H\leq G$ be open and satisfy $[G:H]\leq r$. Clearly $G$ permutes the left cosets of $H$ by left multiplication. Hence we get a well defined map $\phi: G\rightarrow S_r$ where $S_r$ is the permutation group on $r$ objects. Letting $g_1H$ and $g_2H$ be two fixed cosets, notice that $\{x\in G \ | \ xg_1H=g_2H\}=\theta^{-1}(H)$ where $\theta(x)=g_2^{-1}xg_1$, which is open in $G$. As the preimage of any element of $S_r$ is the finite intersection of such sets, one readily sees that $\phi$ is continuous when $S_r$ has the discrete topology.\\

Clearly distinct subgroups will determine distinct $\phi$ maps. Hence it suffices to bound the possible number of such $\phi$ maps. Let $\{\alpha_i\}_{i=1}^s$ be the finite set of topological generators for $G$. Since $\phi$ is continuous, it is completely determined by the values $\{\phi(\alpha_i)\}_{i=1}^s$. As there are $r!$ choices of $\phi(\alpha_i)$ for each $i$, we have only $(r!)^s$ possible choices for $\phi$. Hence there are only finitely many choices of open subgroups $H\leq G$ with index $r\leq m$, finishing the proof.
\end{proof}
We now combine the last several results into one key claim.
\begin{lemma}\label{lemma cyclotomic galois group has fin many bd subgp}
Let $L$ be a field and $d>1$ such that $\text{char}(L)\nmid d$. Then $\text{Gal}(L_d/L)$ has only finitely many open subgroups $H$ of bounded index.
\end{lemma}
\begin{proof}
Notice if  $L_d/L$ is finite, then the conclusion follows. Hence we assume that $L_d/L$ is infinite. By Lemma \ref{lemma cyclotomic extensions finite index}, it suffices to show that $\prod_{i=1}^j \mathbb{Z}_{p_i}^\star$ has only finitely many open subgroups of bounded index. Lemma \ref{lemma p-adic units} and the fact that $\mathbb{Z}_p$ is pro-cyclic imply that $\prod_{i=1}^j \mathbb{Z}_{p_i}^\star$ is topologically finitely generated. By Lemma \ref{lemma fin gen=fin many open subgp}, we are done.
\end{proof}
Lemma \ref{lemma cyclotomic galois group has fin many bd subgp} gives an important proposition.
\begin{proposition}\label{lemma min Gal ext lift over cyclotomic ext}
Let $L$ be a field, $d>1$, and $\{L_j/L\}_{j=1}^\infty$ be a set of distinct minimal Galois extensions. Then $L_j\subset L_d$ for only finitely many $j$. 
\end{proposition}
\begin{proof}
The result follows from Theorem \ref{Finite Simple Gp Bd}, Lemma \ref{lemma cyclotomic galois group has fin many bd subgp}, and  the Galois correspondence. 
\end{proof}
We next derive another simple result about lifting minimal Galois extensions by cyclotomic extensions. To do so, we first prove a related lifting property by Abelian extensions.
\begin{lemma}\label{lemma min gal ext lift over abelian ext}
Let $L$ be a field, $M/L$ be a minimal Galois extension, and $E/L$ be an abelian extension. Then either $M\subset E$ or $ME/E$ is a minimal Galois extension.
\end{lemma}
\begin{proof}
By the Galois correspondence, it suffices to show that $\Gal (ME/E)$ is simple. Recall that $\text{Gal}(ME/E)\cong \text{Gal} (M/(M\cap E))$. Notice that $(M\cap E)/L$ is Galois because $E/L$ is abelian. Since $M/L$ is a minimal Galois extension, we must have $M\cap E$ is either just $L$ or all of $M$. The first case gives that $\Gal (ME/E)\cong \Gal (M/L)$. The second case gives that $\Gal (ME/E)$ is trivial.
\end{proof}

At last, we can provide one more proposition that is essential to the proof of Theorem \ref{main theorem}.
\begin{proposition}\label{prop generic subgps descend+lift}
Let $L$ be a field, $M/L$ be separable, and $f\in L[x]$ a polynomial of degree $d>1$ such that $\text{char}(L)\nmid d$. If $G_f$ over $L$ has infinitely many open maximal normal subgroups, then $G_f$ over $M$ satisfies the same condition.
\end{proposition}
\begin{proof}
Let $M_d:=M\cap L_d$. By the Galois correspondence and Proposition \ref{lemma min Gal ext lift over cyclotomic ext}, there is an infinite sequence $\{L_{j}\}_{j=1}^\infty$ such that $L_{j}\not\subset M_d$ for all $j$. As all cyclotomic extensions are abelian, Lemma \ref{lemma min gal ext lift over abelian ext} implies that $M_d L_{j}/M_d$ is a minimal Galois extension for every $j$. Hence $G_f$ over $M_d$ has infinitely many open maximal normal subgroups. By Corollary \ref{cor roots of unity}, we get that $G_f$ over $M$ does as well.
\end{proof}
Notice that Proposition \ref{prop generic subgps descend+lift} justifies our discussion in the first section comparing the statements of Corollary \ref{cor to main theorm roots of unity included} and Theorem \ref{main theorem}. Now we can justify Theorem \ref{main theorem} in detail.
\begin{proof}[Proof of Theorem \ref{main theorem}] Let $k$ be our given countable Hilbertian field. First notice if $|f^{-\infty}(a)|<\infty$, then we trivially have $[G_f:G_{f,a}]=\infty$ over any field where $f(x)$ is defined (since $|G_{f,a}|<\infty$ and $|G_f|=\infty$). Hence we can assume that $|f^{-\infty}(a)|=\infty$. Define $\mathcal{A}:=\{(F_i, f_i, a_i)\}_{i=1}^\infty$ to consist of the set of ordered field/polynomial/element tuples satisfying the following conditions.
\begin{enumerate}
\item $F_i/k$ is a finite separable extension.\label{Lemma: Fi finite}
\item $f_i\in F_i[x]$ and $a_i\in F_i$.\label{Lemma: fi in Fi}
\item $\deg (f_i)=d_i$ such that $\text{char}(k)\nmid d_i$.\label{Lemma: degrees}
\item $G_i:=G_{f_i}$ over $k_f$ has infinitely many open maximal normal subgroups.  \label{Lemma: main condition}
\item Denote $G_{i,a}:=G_{f_i, a_i}$. We have that ${[G_i:G_{i,a}]<\infty}$ over $F_i$.\label{Lemma: group index}
\end{enumerate}
Let $A$ be the set of $\sigma\in G_{k}$ such that $F^{\bar{\sigma}}$ is a pseudo-algebraically closed Hilbertian field. By Theorems \ref{almost every field is PAC} and \ref{Bary}, $\mu_{k}(A)=1$. If $\mathcal{A}=\emptyset$, set $A=S$ below.\\
 
If $\mathcal{A}$ is nonempty, Proposition \ref{prop generic subgps descend+lift} (or merely the fact that $F_i/k_f$ is finite) along with Conditions (\ref{Lemma: main condition}) and (\ref{Lemma: group index}) imply that $G_{i,a}$ over $F_i$ also has infinitely many open maximal normal subgroups. Denote these by $\{N_{i,j}\}_{j=1}^\infty$.  As each $N_{i,j}$ is maximal among all open normal subgroups, we know the quotient groups $\bar{N}_{i,j}:= G_{i,j}/N_{i,j}$ are simple. Definition \ref{Def Wreath Prod}, Theorem \ref{Subgps of Sd Wreath Prod}, and Condition (\ref{Lemma: degrees}) of $\mathcal{A}$ imply that each $|\bar{N}_{i,j}|$ divides some power of $(d_i)!$ dependent on $j$. Hence a prime divisor of $|\bar{N}_{i, j}|$ for any $j$ must divide $(d_i)!$. This allows us to substitute $\bar{N}_{i,j}$ for $G$ and $(d_i)!$ for $r$ in Theorem \ref{Finite Simple Gp Bd} to obtain a constant $C_i$ such that $|\bar{N}_{i,j}|\leq C_i$ for all $j$. Letting $L_{i,j}/F_i$ denote the minimal Galois fixed field extension corresponding to $N_{i,j}$, we have $[L_{i,j}: F_i]\leq C_i$ for all $j$. By Corollary \ref{cor roots of unity}, $k$ is Hilbertian. We apply Proposition \ref{Prop for main thm} to obtain the desired set $S$.\\

Given $\sigma\in S$, we know $F^{\bar{\sigma}}$ is a pseudo-algebraically closed Hilbertian field. Let $f\in F^{\bar{\sigma}} [x]$ satisfy the assumptions of the theorem. Assume the following for contradiction,
\begin{equation}\label{eq BWOC}
\exists \ a\in F^{\bar{\sigma}} \ \text{such that} \ [G_f: G_{f,a}]<\infty \ \text{over} \  F^{\bar{\sigma}}.    
\end{equation}
Let $K/k$ be the smallest extension containing $a$ over which $f(x)$ is defined.  Proposition \ref{prop group index bound} implies that $[G_f: G_{f,a}]$ over $K$ is also finite.\\

First consider the case that $\mathcal{A}=\emptyset$. Clearly Conditions \ref{Lemma: Fi finite} - \ref{Lemma: degrees} can always be satisfied. If Condition \ref{Lemma: main condition} is always false, then Theorem \ref{main theorem} is vacuously true. This situation will never occur due to our aforementioned example of $f(x)=x^2+a$ not PCF having infinite abelianization (and the fact all induced quadratic extensions are automatically Galois). Hence there will always be (infinitely many) tuples $(F_i, f_i, a_i)$ satisfying Conditions \ref{Lemma: Fi finite} through \ref{Lemma: main condition}. This means Condition \ref{Lemma: group index} must be false for any finite extension $F/k$ and $a\in F$. In particular, $[G_f: G_{f,a}]=\infty$ over $K$. By Proposition \ref{prop group index bound}, we have $[G_f: G_{f,a}]=\infty$ over $F^{\bar{\sigma}}$. The contradiction with Equation \eqref{eq BWOC} that we get finishes the proof in this case.\\

For the rest of the proof, we assume $\mathcal{A}$ is nonempty. Replacing $F_i$ with $K$, $f_i$ with $f$, and $a_i$ with $a$ in the definition of $\mathcal{A}$, Condition \ref{Lemma: Fi finite} holds simply because $F^{\bar{\sigma}}\subset k^{sep}$. Clearly Condition \ref{Lemma: fi in Fi} is true. The veracity of Conditions \ref{Lemma: degrees} and \ref{Lemma: main condition} are given by the assumptions of the theorem. Equation \eqref{eq BWOC} and Proposition \ref{prop group index bound} ensure that Condition \ref{Lemma: group index} is correct.\\

As $\mathcal{A}$ is the set of all such ordered tuples satisfying these conditions, clearly $(K, f, a)=(F_i, f_i, a_i)$ for some $i$. Using the notation in Conditions \ref{Lemma: main condition} and \ref{Lemma: group index}, we can now write $G_i$ for $G_f$ and $G_{i,a}$ for $G_{f,a}$ over $F_i$. For clarity, in the rest of the proof we identify $G_f$ over $F^{\bar{\sigma}}$ with a subgroup $H_{i}$ of $G_{i}$ and analogously for $H_{i,a}$. Since $F_i$ is the smallest extension of $k_\ell$ containing $a_i$ and the coefficients of $f_i(x)$, we know $F_i\subset F^{\bar{\sigma}}$. By Proposition \ref{Prop for main thm}, we have that $L_{i,j}\subset F^{\bar{\sigma}}$ for infinitely many $j$ values, say $j\in J\subset \mathbb{N}$. Let $M_i:=F_i(\zeta_{d_i}^\infty)\cap F^{\bar{\sigma}}$, $J_i$ denote $G_{f}$ over $M_i$, and $J_{i,a}$ denote $G_{f,a}$ over $M_i$. We now have
\begin{equation}\label{eq main pf}
F_{i}^{\bar{\sigma}}:=F^{\bar{\sigma}} \cap M_i(f_i^{-\infty}(a))\supset M_i\prod_{j\in J} L_{i, j}.
\end{equation}
Clearly $H_{i,a}=\text{Gal}(F^{\bar{\sigma}}(f_i^{-\infty}(a))/F^{\bar{\sigma}})$ embeds into $J_{i,a}=\text{Gal}(M_i(f_i^{-\infty}(a))/M_i)$ with an image of index precisely equal to $[F_{i}^{\bar{\sigma}}:M_i]$. By Corollary \ref{cor roots of unity}, $J_i\cong H_i$. Hence
\begin{equation}\label{eq contradiction}
[H_i:H_{i,a}]\geq [J_{i,a}: H_{i,a}]=[F_{i}^{\bar{\sigma}}: M_i]=\infty
\end{equation}
where the last equality follows by Equation \eqref{eq main pf} and Proposition \ref{lemma min Gal ext lift over cyclotomic ext}. Remembering that $G_f=H_i$ and $G_{f,a}=H_{i,a}$ over $F^{\bar{\sigma}}$, we see that Equations \eqref{eq BWOC} and \eqref{eq contradiction} contradict each other. This finishes the proof.
\end{proof}
Corollary \ref{cor to main theorm roots of unity included} follows easily from Theorem \ref{main theorem} and Corollary \ref{cor roots of unity}.
\begin{proof}[Proof of Corollary \ref{cor to main theorm roots of unity included}] Let $k=k(\zeta^\infty)$ be a countable Hilbertian field and $f\in k^{sep}[x]$ satisfy the assumptions of Corollary \ref{cor to main theorm roots of unity included}. Notice that $G_f$ over $k_f$ also has infinitely many open maximal normal subgroups by Corollary \ref{cor roots of unity}. Theorem \ref{main theorem} implies the result.
\end{proof}

\section{Future Directions}
Theorem \ref{main theorem} brings up a natural question.
\begin{question}[Which Polynomials Apply?]\label{which polynomials?}
Given a polynomial $f(x)$ over a countable field $k$, does $G_f$ have infinitely many open maximal normal subgroups?
\end{question}
A general answer to Question \ref{which polynomials?} would be very interesting to pursue, not just to test the applicability of our Theorem \ref{main theorem}, but also because giving an explicit description of $G_{f}$ for an arbitrary polynomial $f(x)$ seems out of reach due to the lengthy computations already required in the aforementioned example of the non-PCF degree 3 case (see \cite{BenedettoCubic}). Our other example where $f$ is a quadratic rational function does not seem any easier (see \cite{BenedettoRational}). Hence such a general result would at least give valuable insight into the underlying structure of general PFIMG.\\

Another subtle point is that all our examples of  polynomials or rational functions satisfying the assumptions of Theorem \ref{main theorem} do so by giving a $G_{f}$ with infinite abelianization. This suggests another question.
\begin{question}\label{infinite abelianization necessary?}(Is Infinite Abelianization Necessary?) Given a polynomial $f(x)$ over a countable field $k$ such that $G_{f}$ has infinitely many open maximal normal subgroups (and so satisfies the conditions of Theorem \ref{main theorem}), is it always true that $G_{f}$ is equal to an iterated wreath product with infinite abelianization?
\end{question}
Notice in order to give a negative answer to Question \ref{infinite abelianization necessary?}, it is not sufficient to find any iterated wreath product that has infinitely many open maximal normal subgroups but does not have infinite abelianization. The reason is that only certain iterated wreath products $G$ can actually be represented by some PFIMG $G_{f}$ corresponding to a polynomial $f(x)$ over a field.\\

It would be intriguing to see whether the conclusion of Theorem \ref{main theorem} extends to polynomials whose PFIMG have infinitely many open maximal subgroups but only finitely many open maximal normal subgroups. This would require working with the fixed fields $F^\sigma$ rather than $F^{\bar{\sigma}}$ (since $\bar{\sigma}$ need not be contained in every open maximal subgroup). Perhaps one could get around this by somehow showing that infinitely many of the open maximal subgroups contain $\bar{\sigma}$ with a Borel-Cantelli style argument. Regarding the linearly disjoint condition of Theorem \ref{BC}, one should recall that two non-Galois minimal extensions over a base field need not be linearly disjoint in general, despite their trivial intersection. Also, the quotient groups are no longer well defined, meaning we cannot guarantee that the indices of the minimal extensions are bounded. Hence one would need to work with pairwise independent open maximal subgroups of bounded index.\\

In replacing $F^{\bar{\sigma}}$ with $F^\sigma$, there is no issue with extending the pseudo-algebraically closed claim (see Jarden's form of the PAC Nullstellensatz Theorem in \cite{Jarden}). The main difficulty in proving such a generalization is that the conclusion of Theorem \ref{Bary} is not known to be true. It is worth mentioning Haran's Bottom Theorem (see \cite{Haran}), which says that $F^\sigma$ is not a finite extension of any proper sub-extension of $L$ for almost every $\sigma\in \text{Gal}(L)$. Additionally, a theorem of Z.M. Chatzidakis says that $F^\sigma$ is not a Galois extension of any proper sub-extension of $L$ for almost every $\sigma\in \text{Gal}(L)$ (see 
\cite{Cha}). In particular, $F^\sigma$ is neither a finite nor an abelian extension of $F^{\bar{\sigma}}$, meaning we cannot use either Serre's fact that finite extensions of Hilbertian fields are Hilbertian (see \cite{Serre}) or Theorem \ref{Abelian Extensions} to argue that any of the $F^\sigma$ fields are Hilbertian.\\

Finally, another natural generalization of our Theorem \ref{main theorem} to attempt would be to allow $\text{char}(k)$ to divide our polynomial degrees or even to consider rational functions. In either case, Theorem \ref{Roots of unity} is known to be false, which our proof of our Theorem \ref{main theorem} depends on. Polynomials with degrees divisible by the positive field characteristic might be purely inseparable over certain specializations, meaning the PFIMG over such $a$ values may be merely the identity element. For rational functions, one needs to differentiate two different associated PFIMG called geometric and arithmetic (see \cite{Ophelia}) that coincide in the polynomial case with $G_f$ over fields $k(\zeta(d^\infty))$ and arbitrary fields $k$, respectively. The analogues of Theorem \ref{Roots of unity} and Corollary \ref{cor roots of unity} do not hold for rational functions, making this case significantly more challenging.

\bibliographystyle{plain} 
\nocite{*}
\bibliography{refs}
\end{document}